\documentclass[11pt]{amsart}
\usepackage{geometry}                
\usepackage{graphicx}
\usepackage{amssymb}
\usepackage{epstopdf}
\usepackage{amsthm}
\usepackage{amsmath}

\usepackage{cancel}

\allowdisplaybreaks

\newtheorem{thm}{Theorem}[section]  
\newtheorem{lem}[thm]{Lemma}  
\newtheorem{cor}[thm]{Corollary}  
\newtheorem{defn}[thm]{Definition}  
\newtheorem{rem}[thm]{Remark}  
\newtheorem{conj}[thm]{Conjecture}

\newcommand{\XX}{\square^3}

\newcommand{\Aut}{\mathrm{Aut}}
\newcommand{\ord}{\mathrm{ord}}
\newcommand{\Gen}{\mathrm{Gen}}
\newcommand{\Mass}{\mathrm{Mass}}

\newcommand{\Z}{\mathbb{Z}}
\newcommand{\Q}{\mathbb{Q}}
\newcommand{\R}{\mathbb{R}}
\newcommand{\N}{\mathbb{N}}
\newcommand{\x}{\vec{x}}
\newcommand{\ve}{\varepsilon}

\newcommand{\al}{\alpha}

\newcommand{\leg}[2]{\left(\frac{#1}{#2}\right)}      
\renewcommand{\(}{\left(}
\renewcommand{\)}{\right)}
\renewcommand{\[}{\left[}
\renewcommand{\]}{\right]}

\title{A proof of the $S$-genus identities for ternary quadratic forms}

\author{Alexander Berkovich}
\address{Department of Mathematics \\ 358 Little Hall, PO Box 110105 \\ University of Florida \\ Gainesville, FL 32611-8105}
\email{alexb@ufl.edu}

\author{Jonathan Hanke}
\address{Department of Mathematics \\ University of Georgia \\ Athens, GA 30602}
\email{jonhanke@math.uga.edu}

\author{Will Jagy}
\address{Math. Sci. Res. Inst. \\ 17 Gauss Way \\ Berkeley, CA 94720-5070}
\email{jagy@msri.org}



\begin{document}
\maketitle

\begin{abstract}
In this paper we prove the main conjectures of Berkovich and Jagy about weighted averages of representation numbers over an $S$-genus of ternary lattices (defined below) for any odd squarefree $S \in \N$.  We do this by reformulating them in terms of local quantities using the Siegel-Weil and Conway-Sloane formulas, and then proving the necessary local identities.  We conclude by conjecturing generalized formulas valid over certain totally real number fields as a direction for future work.
\end{abstract}

\section{Introduction}
In \cite[\S6]{BJ} Berkovich and Jagy propose that the following identity (and also a twisted generalization of it) holds for all odd squarefree natural numbers $S$:
$$
r_{x^2+y^2+z^2}(m) = 48 \!\!\!\!\!\sum_{Q\in \text{$S$-genus}} \frac{r_{Q}(m)}{|\Aut(Q)|}
\qquad\text{for all $m\in\N$ with $m\equiv 1$ or 2$\!\!\!\!\pmod 4$,}
$$
where the {\bf $S$-genus} is defined the set of classes of ternary quadratic forms $Q$ locally equivalent to some quadratic form $B(x,y) + 2Sz^2$ where $B(x,y)$ is a positive definite integer-valued binary quadratic form of discriminant $-8S$.  
This formula is interesting because it gives a precise relationship between an average across certain naturally defined genera of ternary quadratic forms and (the genus of) $x^2 + y^2 + z^2$.  {\it It is not obvious that such an identity should exist, and that it does depends on some intricate relationships among local densities of genera within an $S$-genus.}

In this paper, we prove the formulas \cite[eq (6.3), (6.8) and (6.9)]{BJ} by applying Siegel's product and mass formulas to each genus appearing in the $S$-genus defined above.  We then state a natural generalization of these formulas for totally real number fields $F$ of class number one where $p=2$ is inert.


This paper is an outgrowth of several discussions between Spring 2008 and Fall 2009 about the $S$-genus identities and their relation to the Siegel-Weil formula at the Quadratic Forms and Higher Degree Forms Conferences hosted by Prof. Alladi at the University of Florida (Spring 2008, Spring 2009) and the AMS Southeast Sectional Meeting at Florida Altantic University in Boca Raton, FL (Fall 2009).  This work is also partially supported by the NSA/NSF Grants  and H98230-09-1-0051 and  DMS-0603976 of the first and second authors respectively.

\section{Notation}
Throughout this paper we take $S$ to be an odd squarefree natural number, written $S = p_1 \cdots p_r$ as a product of $r$ distinct primes $p_i\in\N$, and define the $S$-genus as in the Introduction.  We let $p \in \N$  be a prime number, and define the $p$-adic numbers and $p$-adic integers by $\Q_p$ and $\Z_p$ respectively.  We let $v$ denote a place of $\Q$, which is either given as a prime number $p$ (representing the usual $p$-adic absolute value $|x|_p :=p^{-\ord_p(x)}$) or as the archimedean place $\infty$ (representing the real absolute value $|x|$).  At the place $v=\infty$ we adopt the convention that $\Q_\infty := \Z_\infty = \R$.  For an odd prime $p\in \N$, we let $\leg{a}{p}$ denote the usual Legendre symbol.

We define a {\bf quadratic form in $n$ variables over a ring $R$} to be a degree 2 homogeneous polynomial in (exactly) n variables with coefficients in $R$.
We say that two quadratic forms $Q_1$ and $Q_2$ are {\bf equivalent} over a ring $R$, and write $Q_1 \sim_R Q_2$, if there in an invertible change of variables $\x \mapsto A\x$ so that $Q_1(\x) = Q_2(A\x)$ as polynomials.

We define the {\bf classes} of a quadratic form $Q$ to be the equivalence classes of $Q$ for the relation $\sim_\Z$, and the {\bf genus} of $Q$ denoted by $\Gen(Q)$ as the equivalence classes under the simultaneous relations $\sim_{\Z_v}$ for all places $v$ of $\Q$.  We say that a quadratic form $Q$ over $\R$ is {\bf positive (resp. negative) definite} if $Q(\x) >0$ (resp. $<0$) for all $\x \neq \vec 0$ in $\R^n$.  We say that quadratic form over $R$ is {\bf unimodular in $R$} if its Gram determinant is a unit in $R$.

We define the {\bf local representation densities} $\beta_Q, v(m)$ where $v$ is  place of $\Q$ and $Q$ is a $\Z_v$-valued quadratic form (over $\Z_v$) as in \cite[eq (5.1) and (5.2), p368]{Ha}.
%
%
For notational convenience, we occasionally abbreviate the sum of three squares form as $\XX := x^2 + y^2 + z^2$.

\section{Local facts about genera in the $S$-genus}

We begin by establishing several facts about the genera $G$  of quadratic forms contained in a given $S$-genus.  

\begin{lem} \label{lemma:S_genus_local}
The genera $G$ in an $S$-genus are everywhere locally $\Z_p$-equivalent to a diagonal form at all primes $p$.  More explicitly, for $p\neq 2$ we have
$$
G \sim_{\Z_p} 
\begin{cases}
x^2 + y^2 + z^2 & \text{when $p\nmid S$,} \\
\al x^2 + \al p y^2 + pz^2 & \text{when $p\mid S$, for some $\al \in \Z_p^\times /(\Z_p^\times)^2$,}
\end{cases}
$$
and when $p\mid S$ the squareclass $\al$ is uniquely determined.
\end{lem}

\begin{proof}

Since binary forms over $\Z_2$ are completely classified  \cite[Lemma 4.1, pp117-118]{Ca}  and the only primitive non-diagonal forms appearing are $xy$  and $x^2 + xy + y^2$ both of which have odd discriminant, it is not possible for a binary form of even discriminant $-8S$ to be inequivalent to a diagonal form over $\Z_2$.  At all primes $p\neq 2$ the local Jordan splitting theorem \cite[Theorem 3.1, pp115-116]{Ca} guarantees that any (binary) quadratic form can be diagonalized over $\Z_p$, so the condition $Q\in \text{$S$-genus}$ above implies that for every prime $p$ we have 
$$
Q \sim_{\Z_p} \al x^2 + 2S\al y^2 + 2Sz^2
$$ 
for some $\al \in \Z_p^\times /(\Z_p^\times)^2$.

For $p\nmid 2S$, $Q$ is unimodular and the fact that that unimodular diagonal lattices are determined by their determinant squareclass gives that 
\begin{align*}
Q 
&\sim_{\Z_p} 
\al x^2 + 2S\al y^2 + 2Sz^2 \\
&\sim_{\Z_p} 
x^2 + y^2 + 4S^2\al^2z^2 \\
&\sim_{\Z_p} 
x^2 + y^2 + z^2.
\end{align*}

For $p\mid S$, we can simplify the $p$-modular component slightly, giving
\begin{align*}
Q 
&\sim_{\Z_p} 
\al x^2 + 2S\al y^2 + 2Sz^2 \\
&\sim_{\Z_p} 
\al x^2 + \al p y^2 + p (2S/p)^2z^2 \\
&\sim_{\Z_p} 
\al x^2 + \al p y^2 + pz^2. 
\end{align*}
The uniqueness of the squareclass of $\al$ follows from the uniqueness of the one-dimensional unimodular Jordan component (when $e(1) = 0$ in the last line of  \cite[Theorem 3.1, pp115-116]{Ca}).
\end{proof}

We now explicitly describe the genera appearing in a given $S$-genus with an $r$-tuple of signs $\ve_p \in \pm1$.

\begin{defn}
Given a genus $G$ in an $S$-genus, using Lemma  \ref{lemma:S_genus_local} we define the numbers 
$$
\ve_p  := \ve_p(G) := \leg{-\al}{p}  \text{ for all odd primes $p\mid S$.}
$$
These are well-defined by the uniqueness of $\al$ in the last line of Lemma \ref{lemma:S_genus_local}.
\end{defn}

\begin{lem} \label{lemma:S_genus_binary}
The genera $G$ in a fixed $S$-genus
are in 1-to-1 correspondence with the genera of positive definite binary quadratic forms of discriminant $-8S$.  
In particular, there are exactly $2^r$ such binary genera and they are uniquely labelled by the tuples $(\ve_p)_{p\mid S} \in \{\pm1\}^r$.
\end{lem}

\begin{proof}
It is clear from the definition of the $S$-genus that the map $B(x,y) \mapsto B(x,y) + 2Sz^2$ from binary genera of discriminant $-8S$ surjects onto genera in the $S$-genus.  


%
%

To see this map is injective, we first enumerate the local genera of positive definite primitive binary quadratic forms of discriminant $-8S$.  By \cite[\S14.3, Lemmas 3.1, 3.2, 3.3 parts (i-ii), (ii) and (ii$\gamma$)]{Ca} we see that there are exactly two such local genera at each of the places $v \in \{2, p_1, \dots, p_n\}$ and one local genus at all other places.  By \cite[\S14.5, p343, Cor]{Ca} we see that the global genera of this kind are in bijection with the choice of such a local genus $G_v$ at each place subject to the Hasse invariant constraint $\prod_v c_v(G_v) = 1$.  %
%
However if $p\mid 2S$ then the Hasse invariant  for a local diagonal form $Q = ux^2 + pvy^2$ with $u,v \in (\Z_p)^\times$ and $upv=2S$ is $c_p(Q) = (u, pv)_p = (u, pv)_p (u, -u)_p = (u, -2S)_p$.  So $c_p$ takes both $\pm1$ values, because when $p=2$ we know $-2S \not\equiv 1 \pmod 4$, and when $p\neq 2$ we know $-2S \not\in\Z_p^\times$.  Thus the Hasse condition defines an index two subgroup of size $2^r$, for which the local genus $G_2$ is uniquely specified by the choices of the genera at all primes $p_i\mid S$.

Finally, given an $S$-genus $\Gen(Q(x,y) + 2Sz^2)$ we can recover 
the values $\ve_{p_i}(Q) := \leg{-\al}{p_i}$ of 
the underlying genus of binary forms $\Gen(Q)$ because they both represent 
the same values $\al$ (mod $S$),
showing the desired injectivity.
\end{proof}


%


\begin{rem}
Our numbers $\ve_p$ are just another way of expressing the $r$ classical genus characters $\chi_i(m) := \leg{m}{p_i} = \leg{-1}{p_i}\ve_{p_i}$ where $m$ is an number represented by a form in the genus that is relatively prime to $2S$.  (See for example \cite[p222-224]{Cohn} or \cite[\S4.1]{Buell}.)
%
%
\end{rem}


\section{Computing the local representation densities}

In this section we compute the local representation densities $\beta_{x^2 + y^2 + z^2, p}(m)$ and $\beta_{G, p}(m)$ for a genus $G$ contained in an $S$-genus.

\begin{lem} \label{lemma:three_squares_odd_densities}
Suppose that $p$ is an odd prime.  Then
$$
\beta_{x^2 + y^2 + z^2, p}(m) 
= 
\begin{cases}
\(1-\frac{1}{p^2}\)(1 + \frac{1}{p} + \cdots + \frac{1}{p^{k-1}}) + \frac{1}{p^k}\(1 + \frac{\leg{-m/p^{2k}}{p}}{p}\) 
&\text{ if $\ord_p(m) = 2k$,}\\
\(1-\frac{1}{p^2}\)(1 + \frac{1}{p} + \cdots + \frac{1}{p^{k}}) &\text{ if $\ord_p(m) = 2k+1$,}\\
\end{cases}
$$
for some $k\in \Z\geq 0$.
\end{lem}

\begin{proof}
From \cite[eq (5.2), Remark 3.4.1(b), and Table 1]{Ha} we see 
for $p\neq 2$ that
$$
\beta_{x^2 + y^2 + z^2, p}^{Good \,\cup\, Bad}(m)
= \beta_{x^2 + y^2 + z^2, p}^{Good}(m)
= 
\begin{cases}
\frac{p^2 + p\leg{-m}{p}}{p^2} & \text{ if $p\nmid m$,}\\
\frac{p^2 - 1}{p^2} & \text{ if $p\mid m$.}
\end{cases}
$$
The lemma follows from this and the formula 
\begin{equation} \label{eq:Zero-type_recusion}
\begin{aligned}
\beta_Q(m) 
&= \beta^{Good \,\cup\, Bad}_Q(m) + \beta^{Zero}_Q(m) \\
&= \beta^{Good \,\cup\, Bad}_Q(m) + \tfrac{1}{p}\beta_Q(m/p^2)
\end{aligned}
\end{equation}
from \cite[Remark 3.4.1(d) about $\pi_Z$, p362]{Ha} for ternary forms.  This result is also mentioned in \cite[\S28.4 with $m=3$, p228]{Siegel}.
\end{proof}

\begin{lem} \label{lemma:S_genus_densities}
Suppose that $G$ is a genus in an $S$-genus and that $p\mid S$ is an odd prime.  
Then
$$
\beta_{G, p}(m) 
= 
\begin{cases}
\(1-\frac{1}{p}\)(1+ \ve_p)(1 + \frac{1}{p} + \cdots + \frac{1}{p^{k-1}}) + \frac{1}{p^k}\(1 + \ve_p\leg{-m/p^{2k}}{p}\) 
&\text{ if $\ord_p(m) = 2k$,}\\
\(1-\frac{1}{p}\)(1+ \ve_p)(1 + \frac{1}{p} + \cdots + \frac{1}{p^{k-1}}) + \frac{1}{p^k}\(1 - \frac{\ve_p}{p}\) 
&\text{ if $\ord_p(m) = 2k+1$,}\\
\end{cases}
$$
for some $k\in \Z\geq 0$.
\end{lem}

\begin{proof}


For odd primes $p\mid S$ we can compute the representation densities $\beta_{Q, p}(m)$ explicitly via the formulas \cite[\S3]{Ha}, giving
\begin{align*}
\beta_{\al x^2+\al py^2+pz^2, p}(m) 
&= \beta^{Good}_{\al x^2+\al py^2+pz^2, p}(m) + \beta^{Bad I}_{\al x^2+\al py^2+pz^2, p}(m) 
\end{align*}
because the highest power of $p$ dividing a coefficient is 1.   From \cite[Table 1, p363]{Ha} we compute
\begin{equation*}
\beta^{Good}_{\al x^2+\al py^2+pz^2, p}(m) = 
\beta^{Good}_{\al x^2, p}(m) = 
\begin{cases}
1 + \ve_p \leg{-m}{p} & \text{ if } p\nmid m,\\
0 & \text{ if } p\mid m,\\
\end{cases}
\end{equation*}
and also \cite[p360, middle, Bad-type I reduction map]{Ha} gives
\begin{equation*}
\beta^{Bad I}_{\al x^2+\al py^2+pz^2, p}(m) = 
\begin{cases}
0 & \text{ if } p\nmid m,\\
\beta^{Good}_{\al px^2+\al y^2+ z^2, p}(m/p) & \text{ if } p\mid m,\\
\end{cases}
\end{equation*}
where
\begin{equation*}
\beta^{Good}_{\al px^2+\al y^2+ z^2, p}(m/p)
= \beta^{Good}_{\al y^2+ z^2, p}(m/p)
= \begin{cases}
1 - \frac{\ve_p}{p} & \text{ if } p\nmid m,\\
\(1 - \frac{1}{p}\)\(1 + \ve_p\) & \text{ if } p\mid m.\\
\end{cases}
\end{equation*}
These combine to give
\begin{equation*}
\beta^{Good \,\cup\, Bad}_{\al x^2+\al py^2+pz^2, p}(m) =
\begin{cases}
1 + \ve_p \leg{-m}{p} & \text{ if } p\nmid m,\\
1 - \frac{\ve_p}{p} & \text{ if } \ord_p(m) = 1,\\
\(1 - \frac{1}{p}\)\(1 + \ve_p \) & \text{ if } \ord_p(m) \geq 2.\\
\end{cases}
\end{equation*}
and the lemma follows by combining these with equation (\ref{eq:Zero-type_recusion}).
\end{proof}

\begin{lem} \label{lemma:three_squares_two_density}
When $p=2$, we have the following local representation densities:
\begin{equation}
\beta_{x^2+y^2+z^2, 2}(m) 
= \begin{cases}
\frac{3}{2} & \text{ if } m\equiv 1,2 \pmod 4, \\
1 & \text{ if } m\equiv 3 \pmod 8, \\
0 & \text{ if } m\equiv 7 \pmod 8, \\
\end{cases}
\end{equation}
and 
\begin{equation}
\beta_{G, 2}(m) 
= \beta_{ux^2+2uy^2+2z^2, 2}(m)
= \begin{cases}
1 & \text{ if } m\equiv 1,2 \pmod 4, \\
0 \text{ or } 2 & \text{ if $m\equiv 3 \pmod 4$ (depending on $u$)}. \\
\end{cases}
\end{equation}
\end{lem}

\begin{proof}
At $p=2$ we compute the two relevant local densities for $m$ with $\ord_2(m) \leq 1$ by counting all representation numbers modulo $16$. (This is sufficient by \cite[Lemma 3.2 and the Bad-type I reduction map on p360 middle]{Ha}).
\end{proof}

\begin{lem} \label{lemma:archimedean_density}
The local representation density at $v=\infty$ has the form
$$
\beta_{Q, \infty}(m) = C_n \det(Q)^{-\frac12} m^\frac{n-2}{2}
$$
where $C_n$ is some constant depending only on $n := \dim(Q)$.
\end{lem}

\begin{proof}
This is proved in \cite[(6.40) with $m=3$ and $n=1$, p42]{Siegel} and stated more recently in \cite[eq (5.5), p369]{Ha}.
\end{proof}

\section{Computing masses of genera in an $S$-genus}

In this section we prove the mass conjecture for $S$-genera formulated in  \cite[eq (6.3)]{BJ} by using an explicit general mass formula of Conway and Sloane \cite{CS}.

\begin{lem} \label{lemma:mass_formula}
  Suppose that $G$ is a genus in an $S$-genus.  Then
$$
\Mass(G) :=
\sum_{Q\in G} \frac{1}{|\Aut(Q)|} 
= \frac{1}{16}
\cdot \prod_{p\mid S} \frac{p+\ve_p}{2}.
$$
\end{lem}

\begin{proof}  From the the mass formula \cite[eq (2)]{CS} of Conway and Sloane, the mass of a (ternary) genus $G$ in the given $S$-genus is
\begin{align*}
\Mass(G)
&= 2 \pi^\frac{-n(n+1)}{4} \prod_{j=1}^n \Gamma(j/2) \prod_p (2m_p(G)) \\
&= 2 \pi^{-3} \cdot \sqrt{\pi} \cdot 1 \cdot \frac{\sqrt{\pi}}{2}  \prod_p (2m_p(G)) \\
&= \frac{1}{\pi^2} \prod_p (2m_p(G)) 
\end{align*} 
where the local masses $m_p(G)$ are defined in \cite[eq (3)]{CS} by 
\begin{equation*}
m_p(G) := \prod_q M_q(G) \cdot \prod_{q < q'} \(q'/q\)^\frac{n(q) \cdot n(q')}{2} \cdot 2^\text{$n$(I,I) - $n$(II)}.
\end{equation*}

From (\cite[Table 1]{CS}) we see that there is a unique ``species'' when the unimodular blocks are odd, so for $p\nmid 2S$ we have $2m_p = 2M_1 = (1-\frac{1}{p^2})^{-1}$.  For $p\neq 2$ with $p\mid 2S$ we have 
\begin{equation*}
2m_p 
= 2 \cdot M_1 \cdot M_p \cdot p^\frac{1\cdot 2}{2} 
= 2 \cdot \frac{1}{2} \cdot \frac{1}{2}\(1-\frac{\leg{-\alpha}{p}}{p}\)^{-1} \cdot p 
= \frac{p}{2\(1-\frac{\leg{-\alpha}{p}}{p}\)} 
\end{equation*}

Finally when $p=2$ we have four standard mass factors $M_{1/2}, M_{1}, M_{2}$, and $M_{4}$  of types II${}_0$, I${}_1$, I${}_2,$ and II${}_0$ respectively.  Since all of these forms have an odd form adjacent to it, they are ``bound'' (in the notation of \cite{CS}) and so they have ``species'' 1 and contribute a factor of $\frac{1}{2}$ each.  Therefore we have 
\begin{align*}
2m_2(G) 
&= 2\prod_q M_q(G) \cdot \prod_{q < q'} \(q'/q\)^\frac{n(q) \cdot n(q')}{2} \cdot 2^{n(I,I) - n(II)} \\
&= 2\(\prod_{q\in \{\frac{1}{2}, 1, 2, 4\}}\frac{1}{2}\) \cdot \(2/1\)^\frac{1 \cdot 2}{2} \cdot 2^{1 - 0}
\quad=\quad \frac{1}{2} .
\end{align*}

Putting these all together gives 
\begin{align*}
\mathrm{Mass}(G)
&= \frac{1}{\pi^2} \prod_p (2m_p(G)) \\
&= \frac{1}{\pi^2} (2m_2(G))
 \cdot \prod_{2\neq p\mid 2S} (2m_p(G))
  \cdot \prod_{p\nmid 2S} (2m_p(G)) \\
&= \frac{1}{\pi^2} (2m_2(G))
 \cdot \prod_{2\neq p\mid 2S} \frac{p}{2\(1-\frac{\ve_p}{p}\)}
 \cdot \prod_{p\nmid 2S}  \frac{1}{1-\frac{1}{p^2}}\\
&= \frac{1}{\pi^2} (2m_2(G))
 \cdot \prod_{2\neq p\mid 2S} \frac{p}{2\cancel{\(1-\frac{\ve_p}{p}\)}}  
  \frac{\cancel{\(1-\frac{\ve_p}{p}\)}\(1+\frac{\ve_p}{p}\)}{1-\frac{1}{p^2}}
 \cdot \prod_{p\nmid 2S}  \frac{1}{1-\frac{1}{p^2}}\\
&= \frac{1}{\pi^2} (2m_2(G))
 \cdot \prod_{2\neq p\mid 2S} \frac{p+ \ve_p}{2}
 \cdot \prod_{p\neq 2}  \frac{1}{1-\frac{1}{p^2}}\\
&= \frac{1}{\pi^2} (2m_2(G))
 \cdot \prod_{2\neq p\mid 2S} \frac{p+ \ve_p}{2}
 \cdot \frac{\pi^2}{6} \cdot \frac{3}{4}\\
&= \frac{1}{8} (2m_2(G))
 \cdot \prod_{2\neq p\mid 2S} \frac{p+ \ve_p}{2} \\
&= \frac{1}{16}
 \cdot \prod_{2\neq p\mid 2S} \frac{p+ \ve_p}{2}.
\end{align*} 
%
%
%
%
\end{proof}

\section{The main theorem}
In this section we prove the following theorem and corollary (by taking $W=1$), which were formulated and conjectured in \cite[eq (6.9) and eq (6.8)]{BJ}:

\begin{thm}  \label{Thrm:main_theorem}
Suppose that $S \in \N$ is an odd squarefree number and $W \in \N$ divides $S$.  Then
\begin{equation}\label{eq:main_thrm}
\sum_{Q\in \text{$S$-genus}} \epsilon_W(Q) \frac{r_{Q}(m)}{|\Aut(Q)|} 
= W \, \frac{r_{x^2 + y^2 + z^2}(m/W^2)}{48}
\end{equation}
for all $m \in \N$ with $W \mid m$ and
 $m \equiv 1,2 \pmod 4$, where  $\epsilon_W(Q)$ is defined by the formula
 $$
\epsilon_W(Q) := \prod_{p\mid W} \ve_p(\Gen(Q)).
 $$ 
 (Here if $W^2 \nmid m$ then both sides of (\ref{eq:main_thrm}) are zero.)
\end{thm}


\begin{cor}  Suppose that $S \in \N$ is an odd squarefree number.  Then
$$
\sum_{Q\in \text{$S$-genus}} \frac{r_{Q}(m)}{|\Aut(Q)|} 
= \frac{r_{x^2 + y^2 + z^2}(m)}{48}
$$
for all $m \in \N$ with 
 $m \equiv 1,2 \pmod 4$.
\end{cor}

\begin{proof}
In what follows we occasionally write $\XX := x^2 + y^2 + z^2$ for brevity. 
We begin by expressing the LHS locally, and trying to relate it locally to the RHS.  Since the class number of $x^2 + y^2 + z^2$ is one, we can recover the RHS globally from its local expressions.

By the Siegel-Weil formula for definite ternary quadratic forms \cite[\S5]{Rallis} we have that 
{
\allowdisplaybreaks
\begin{align}
\sum_{Q\in \text{$S$-genus}} \epsilon_W(Q) \frac{r_{Q}(m)}{|\Aut(Q)|}
&= \sum_{\substack{\text{(genera)} \\ G\in \text{$S$-genus}}} \epsilon_W(G) \sum_{Q'\in G} \frac{r_{Q'}(m)}{|\Aut(Q')|} \\
&= \sum_{G \in \text{$S$-genus}} \(\prod_{p\mid W} \ve_p\) \Mass(G) \prod_v \beta_{G, v}(m) \\
&= \sum_{G \in \text{$S$-genus}} 
\(\Mass(G)  \prod_{p\mid W} \ve_p \prod_{v \mid 2S\infty} \beta_{G, v}(m) \prod_{p \nmid 2S} \beta_{G, p}(m) \).
\end{align}
}

Since at all primes $p\nmid 2S$ we have $G \sim_{\Z_p} x^2 + y^2 + z^2$, we can separate out a product which looks very much like the product of local densities for $x^2 + y^2 + z^2$, giving
{
\allowdisplaybreaks
\begin{align*}
&= \sum_{G \in \text{$S$-genus}} 
\Mass(G) \(\prod_{p\mid W} \ve_p\) \(\prod_{v \mid 2S\infty} \beta_{G, v}(m)\) \prod_{p \nmid 2S} \beta_{\XX, p}(m) \\
&= \(\prod_{p \nmid 2S} \beta_{\XX, p}(m)\)
\sum_{G \in \text{$S$-genus}} \Mass(G) \(\prod_{p\mid W} \ve_p\) \prod_{v \mid 2S\infty} \beta_{G, p}(m)\\
&= \( \det(Q)^{-1/2}\beta_{\XX, \infty}(m) \prod_{p \nmid 2S} \beta_{\XX, p}(m)\)
\sum_{G \in \text{$S$-genus}} \Mass(G) \(\prod_{p\mid W} \ve_p\) \prod_{p \mid 2S} \beta_{G, p}(m)\\
&= \( \frac{1}{2S}\beta_{\XX, \infty}(m) \prod_{p \nmid 2S} \beta_{\XX, p}(m)\)
\sum_{G \in \text{$S$-genus}} \Mass(G) \(\prod_{p\mid W} \ve_p\) \prod_{p \mid 2S} \beta_{G, p}(m)\\
&= \( \frac{W}{2S}\beta_{\XX, \infty}(m/W^2) \prod_{p \nmid 2S} \beta_{\XX, p}(m/W^2)\)
\sum_{G \in \text{$S$-genus}} \Mass(G) \(\prod_{p\mid W} \ve_p\) \prod_{p \mid 2S} \beta_{G, p}(m)\\
&=W  \(\prod_{v}\beta_{\XX, v}(m/W^2) \)
\( \frac{1}{2S} \prod_{p \mid 2S} \frac{1}{\beta_{\XX, p}(m/W^2)} \)
\sum_{G \in \text{$S$-genus}} \Mass(G) \(\prod_{p\mid W} \ve_p\) \prod_{p \mid 2S} \beta_{G, p}(m)\\
&= \(\frac{W}{48}\prod_{v}\beta_{\XX, v}(m/W^2) \)
\(48 \sum_{G\in \text{$S$-genus}} 
\frac{\Mass(G)}{2S} 
\(\prod_{p\mid W} \frac{\ve_p \,\beta_{G, p}(m)}{\beta_{\XX, p}(m/p^2)}\)
\prod_{\substack{p \mid 2S \\ p\nmid W}} \frac{\beta_{G, p}(m)}{\beta_{\XX, p}(m)} \).
\end{align*}
}
(Here we note that technically the division by $ \beta_{\XX, p}(m/p^2)$ in the next to last equality  may be undefined if $\beta_{\XX, p}(m/p^2) = 0$, however in this case Lemma \ref{lemma:cases} ensures that both sides of (\ref{eq:main_thrm}) are zero.)
%

Because  $\Mass(x^2 + y^2 + z^2) = \frac{1}{48}$ and $x^2 + y^2 + z^2$ has class number one, from Siegel's product formulas we see that the desired  equality 
\begin{equation}
\sum_{Q\in \text{$S$-genus}} \epsilon_W(G) \frac{r_{Q}(m)}{|\Aut(Q)|} 
= W \cdot \Mass(x^2 + y^2 + z^2) \prod_{v}\beta_{x^2+y^2+z^2, v}(m)
= W \frac{r_{x^2 + y^2 + z^2}(m)}{48} 
\end{equation}
is equivalent to showing that the expression
\begin{equation}
E := 48 \sum_{G\in \text{$S$-genus}} 
\frac{\Mass(G)}{2S}
\(\prod_{p\mid W} \frac{\ve_p \,\beta_{G, p}(m)}{\beta_{x^2+y^2+z^2, p}(m/p^2)}\)
\prod_{\substack{p \mid 2S \\ p\nmid W}} \frac{\beta_{G, p}(m)}{\beta_{x^2+y^2+z^2, p}(m)} 
\end{equation}
is equal to one.  By substituting our mass formula in Lemma \ref{lemma:mass_formula} and using the fact that $2S$ is square-free, we can rewrite $E$ more locally as
\begin{equation}
E = 3 \sum_{G\in \text{$S$-genus}} 
\frac{\beta_{G, 2}(m)}
{2\cdot \beta_{x^2+y^2+z^2, 2}(m)}
\[
\prod_{p \mid W} \frac{\ve_p\, \beta_{G, p}(m) \(p + \ve_p\)}
{2p\cdot \beta_{x^2+y^2+z^2, p}(m/p^2)}
\]
\[
\prod_{\substack{p \mid S \\ p\nmid W}} \frac{\beta_{G, p}(m) \(p + \ve_p\)}
{2p\cdot \beta_{x^2+y^2+z^2, p}(m)}
\]. 
\end{equation}
From Lemma \ref{lemma:three_squares_two_density}, we can evaluate the local densities at $p=2$, giving
\begin{equation} \label{eq:E}
E = \cancel3 \sum_{G\in \text{$S$-genus}} 
\frac{1}
{\cancel2 \cdot \frac{\cancel3}{\cancel2}}
\[
\prod_{p \mid W} \frac{\ve_p\, \beta_{G, p}(m) \(p + \ve_p\)}
{2p\cdot \beta_{x^2+y^2+z^2, p}(m/p^2)}
\]
\[
\prod_{\substack{p \mid S \\ p\nmid W}} \frac{\beta_{G, p}(m) \(p + \ve_p\)}
{2p\cdot \beta_{x^2+y^2+z^2, p}(m)}
\]. 
\end{equation}

Because the genera $G$ are indexed by the $r$-tuple of values $(\ve_p)_{p\mid S} \in \{\pm1\}^r$ and the since the factors in the product are all independent, we have that
{
\allowdisplaybreaks
\begin{align}
E 
&= 
\sum_{G\in \text{$S$-genus}}
\[
\prod_{p \mid W} \frac{\ve_p\, \beta_{G, p}(m) \(p + \ve_p\)}
{2p\cdot \beta_{x^2+y^2+z^2, p}(m/p^2)}
\]
\[
\prod_{\substack{p \mid S \\ p\nmid W}} \frac{\beta_{G, p}(m) \(p + \ve_p\)}
{2p\cdot \beta_{x^2+y^2+z^2, p}(m)}
\] \\ 
&=
\sum_{(\ve_p)_{p\mid S} \in \{\pm 1\}^{r}} 
\[
\prod_{p \mid W} \frac{\ve_p\, \beta_{G, p}(m) \(p + \ve_p\)}
{2p\cdot \beta_{x^2+y^2+z^2, p}(m/p^2)}
\]
\[
\prod_{\substack{p \mid S \\ p\nmid W}} \frac{\beta_{G, p}(m) \(p + \ve_p\)}
{2p\cdot \beta_{x^2+y^2+z^2, p}(m)}
\] \\ 
&=  
\[
\prod_{p \mid W} \sum_{\ve_p \in \{\pm 1\}}
\frac{\ve_p\, \beta_{G, p}(m) \(p + \ve_p\)}
{2p\cdot \beta_{x^2+y^2+z^2, p}(m/p^2)}
\]
\[
\prod_{\substack{p \mid S \\ p\nmid W}} \sum_{\ve_p \in \{\pm 1\}}
\frac{\beta_{G, p}(m) \(p + \ve_p\)}
{2p\cdot \beta_{x^2+y^2+z^2, p}(m)}
\].
\end{align}
}

By Lemma \ref{lemma:cases} we see that all summands for odd primes are identically one, giving (formally that) $E = \prod_{p \mid S} 1 = 1$, which completes the proof.

\end{proof}

We now evaluate the individual summands appearing in the last step of the proof of Theorem \ref{Thrm:main_theorem} for (odd) primes $p\mid S$ using our previously computed local density formulas. 

\begin{lem}  \label{lemma:cases}
Suppose $G$ is a genus in an $S$-genus, and $p\mid S$ is an (odd) prime.  Then
$$
\sum_{\ve_p \in \{\pm 1\}}
\frac{\beta_{G, p}(m) \(p + \ve_p\)}
{2p} =  \beta_{x^2+y^2+z^2, p}(m),
$$
and 
$$
\sum_{\ve_p \in \{\pm 1\}}
\frac{\ve_p \, \beta_{G, p}(m) \(p + \ve_p\)}{2p}
= \beta_{x^2+y^2+z^2, p}(m/p^2).
$$
\end{lem}

\begin{proof}
In the following proof we verify several algebraic identities of polynomials in $p$, which justifies our neglecting the issue of division by zero.

{\bf Case 1: $\ord_p(m) = 2k$ is even:}
{
\allowdisplaybreaks
\begin{align*}
\sum_{\ve_p \in \{\pm 1\}}
&\frac{\beta_{G, p}(m) \(p + \ve_p\)}
{2p\cdot \beta_{x^2+y^2+z^2, p}(m)} \\
&=
\frac{\sum_{\ve_p \in \{\pm 1\}}\[ \(1-\frac{1}{p}\)(1+ \ve_p)(1 + \frac{1}{p} + \cdots + \frac{1}{p^{k-1}}) + \frac{1}{p^k}\(1 + \ve_p\leg{-m}{p}\) \]  \(1 + \frac{\ve_p}{p}\)}
{2\cdot \[ \(1-\frac{1}{p^2}\)(1 + \frac{1}{p} + \cdots + \frac{1}{p^{k-1}}) + \frac{1}{p^k}\(1 + \frac{\leg{-m}{p}}{p}\) \]} 
\end{align*}
}
When $\ve_p=-1$ the first term of the numerator vanishes, giving
{
\allowdisplaybreaks
\begin{align*}
&=
\frac{
 2\(1-\frac{1}{p^2}\)(1 + \frac{1}{p} + \cdots + \frac{1}{p^{k-1}}) 
 + \frac{1}{p^k} \sum_{\ve_p \in \{\pm 1\}}
  \(1 + \frac{\ve_p}{p}\)
  \(1 + \ve_p\leg{-m}{p}\)
}
{2\cdot \[ \(1-\frac{1}{p^2}\)(1 + \frac{1}{p} + \cdots + \frac{1}{p^{k-1}}) 
+ \frac{1}{p^k}\(1 + \frac{\leg{-m}{p}}{p}\) \]} \\
&=
\frac{
 2\(1-\frac{1}{p^2}\)(1 + \frac{1}{p} + \cdots + \frac{1}{p^{k-1}}) 
 + \frac{1}{p^k} \sum_{\ve_p \in \{\pm 1\}}
  \(1 + \frac{\ve_p}{p} + \ve_p\leg{-m}{p} + \frac{\leg{-m}{p}}{p}\)
}
{2\cdot \[ \(1-\frac{1}{p^2}\)(1 + \frac{1}{p} + \cdots + \frac{1}{p^{k-1}}) 
+ \frac{1}{p^k}\(1 + \frac{\leg{-m}{p}}{p}\) \]} \\
&=
\frac{
 2\(1-\frac{1}{p^2}\)(1 + \frac{1}{p} + \cdots + \frac{1}{p^{k-1}}) 
 + \frac{1}{p^k} \(2 + 0 + 0 + 2\frac{\leg{-m}{p}}{p}\)
}
{2\cdot \[ \(1-\frac{1}{p^2}\)(1 + \frac{1}{p} + \cdots + \frac{1}{p^{k-1}}) 
+ \frac{1}{p^k}\(1 + \frac{\leg{-m}{p}}{p}\) \]}  
\quad = \quad 1.
\end{align*}
}

{\bf Case 2: $\ord_p(m) = 2k+1$ is odd:}
{
\allowdisplaybreaks
\begin{align*}
\sum_{\ve_p \in \{\pm 1\}}
&\frac{\beta_{G, p}(m) \(p + \ve_p\)}
{2p\cdot \beta_{x^2+y^2+z^2, p}(m)} \\
&=
\frac{\sum_{\ve_p \in \{\pm 1\}}\[ \(1-\frac{1}{p}\)(1+ \ve_p)(1 + \frac{1}{p} + \cdots + \frac{1}{p^{k-1}})
 + \frac{1}{p^k}\(1 - \frac{\ve_p}{p}\) \]  \(1 + \frac{\ve_p}{p}\)}
{2\cdot \[ \(1-\frac{1}{p^2}\)(1 + p + \cdots + p^{k})  \]} 
\end{align*}
}
When $\ve_p=-1$ the first term of the numerator vanishes, giving
{
\allowdisplaybreaks
\begin{align*}
&=
\frac{
 2\(1-\frac{1}{p^2}\)(1 + \frac{1}{p} + \cdots + \frac{1}{p^{k-1}}) 
 + \frac{1}{p^k} \sum_{\ve_p \in \{\pm 1\}}
  \(1 - \frac{\ve_p}{p}\)
  \(1 + \frac{\ve_p}{p}\)
}
{2\cdot \[ \(1-\frac{1}{p^2}\)(1 + \frac{1}{p} + \cdots + \frac{1}{p^{k}})  \]} \\
&=
\frac{
 2\(1-\frac{1}{p^2}\)(1 + \frac{1}{p} + \cdots + \frac{1}{p^{k-1}}) 
 + \frac{1}{p^k} \sum_{\ve_p \in \{\pm 1\}}
  \(1 - \frac{1}{p^2} \)
}
{2\cdot \[ \(1-\frac{1}{p^2}\)(1 + \frac{1}{p} + \cdots + \frac{1}{p^{k}}) \]} \\
&=
\frac{
 2\(1-\frac{1}{p^2}\)(1 + \frac{1}{p} + \cdots + \frac{1}{p^{k-1}}) 
 + \frac{2}{p^k}   \(1-\frac{1}{p^2}\)
}
{2\cdot \[ \(1-\frac{1}{p^2}\)(1 + \frac{1}{p} + \cdots + \frac{1}{p^{k}})   \]}  
\quad = \quad 1.
\end{align*}
}

{\bf Case 3: $\ord_p(m) = 2k$ is even with $k\geq 1$:}
{
\allowdisplaybreaks
\begin{align*}
\sum_{\ve_p \in \{\pm 1\}}
& \frac{\ve_p\, \beta_{G, p}(m) \(p + \ve_p\)}
{2p\cdot \beta_{x^2+y^2+z^2, p}(m/p^2)} \\
&=
\frac{\sum_{\ve_p \in \{\pm 1\}} \ve_p \[ \(1-\frac{1}{p}\)(1+ \ve_p)(1 + \frac{1}{p} + \cdots + \frac{1}{p^{k-1}}) + \frac{1}{p^k}\(1 + \ve_p\leg{-m}{p}\) \]  \(1 + \frac{\ve_p}{p}\)}
{2\cdot \[ \(1-\frac{1}{p^2}\)(1 + \frac{1}{p} + \cdots + \frac{1}{p^{k-2}}) + \frac{1}{p^{k-1}}\(1 + \frac{\leg{-m}{p}}{p}\) \]} 
\end{align*}
}
When $\ve_p=-1$ the first term of the numerator vanishes, giving
{
\allowdisplaybreaks
\begin{align*}
&=
\frac{
 2\(1-\frac{1}{p^2}\)(1 + \frac{1}{p} + \cdots + \frac{1}{p^{k-1}}) 
 + \frac{1}{p^k} \sum_{\ve_p \in \{\pm 1\}}
  \ve_p
  \(1 + \frac{\ve_p}{p}\)
  \(1 + \ve_p\leg{-m}{p}\)
}
{2\cdot \[ \(1-\frac{1}{p^2}\)(1 + \frac{1}{p} + \cdots + \frac{1}{p^{k-2}}) + \frac{1}{p^{k-1}}\(1 + \frac{\leg{-m}{p}}{p}\) \]} \\
&=
\frac{
 2\(1-\frac{1}{p^2}\)(1 + \frac{1}{p} + \cdots + \frac{1}{p^{k-1}}) 
 + \frac{1}{p^k} \sum_{\ve_p \in \{\pm 1\}}
  \(\ve_p + \frac{1}{p} + \leg{-m}{p} + \ve_p\frac{\leg{-m}{p}}{p}\)
}
{2\cdot \[ \(1-\frac{1}{p^2}\)(1 + \frac{1}{p} + \cdots + \frac{1}{p^{k-2}}) 
+ \frac{1}{p^{k-1}} \(1 + \frac{\leg{-m}{p}}{p}\) \]} \\
&=
\frac{
 \cancel2 \(1-\frac{1}{p^2}\)(1 + \frac{1}{p} + \cdots + \frac{1}{p^{k-1}}) 
 + \frac{1}{p^k} \(0 + \frac{\cancel 2}{p} + \cancel 2 \leg{-m}{p} + 0\)
}
{\cancel2\cdot \[ \(1-\frac{1}{p^2}\)(1 + \frac{1}{p} + \cdots + \frac{1}{p^{k-2}}) + \frac{1}{p^{k-1}} \(1 + \frac{\leg{-m}{p}}{p}\) \]} \\
&=
\frac{
 \(1-\frac{1}{p^2}\)(1 + \frac{1}{p} + \cdots + \frac{1}{p^{k-2}}) + \frac{1}{p^{k-1}}\(1-\cancel{\frac{1}{p^2}}\)
 + \frac{1}{p^k} \(\cancel{\frac{1}{p}} + \leg{-m}{p} \)
}
{ \[ \(1-\frac{1}{p^2}\)(1 + \frac{1}{p} + \cdots + \frac{1}{p^{k-2}}) + \frac{1}{p^{k-1}} \(1 + \frac{\leg{-m}{p}}{p}\) \]}  
\quad 
= \quad 1.
\end{align*}
}

{\bf Case 4: $\ord_p(m) = 2k+1$ is odd with $k\geq 1$:}
{
\allowdisplaybreaks
\begin{align*}
\sum_{\ve_p \in \{\pm 1\}}
&\frac{\ve_p \, \beta_{G, p}(m) \(p + \ve_p\)}
{2p\cdot \beta_{x^2+y^2+z^2, p}(m/p^2)} \\
&=
\frac{\sum_{\ve_p \in \{\pm 1\}}\ve_p  \[ \(1-\frac{1}{p}\)(1+ \ve_p)(1 + \frac{1}{p} + \cdots + \frac{1}{p^{k-1}})
  + \frac{1}{p^k}\(1 - \frac{\ve_p}{p}\) \]  \(1 + \frac{\ve_p}{p}\)}
{2\cdot \[ \(1-\frac{1}{p^2}\)(1 + \frac{1}{p} + \cdots + \frac{1}{p^{k-1}})  \]} 
\end{align*}
}
When $\ve_p=-1$ the first term of the numerator vanishes, giving
{
\allowdisplaybreaks
\begin{align*}
&=
\frac{
 2\(1-\frac{1}{p^2}\)(1 + \frac{1}{p} + \cdots + \frac{1}{p^{k-1}})
  + \frac{1}{p^k} \sum_{\ve_p \in \{\pm 1\}} \ve_p 
  \(1 - \frac{\ve_p}{p}\)
  \(1 + \frac{\ve_p}{p}\)
}
{2\cdot \[ \(1-\frac{1}{p^2}\)(1 + \frac{1}{p} + \cdots + \frac{1}{p^{k-1}})  \]} \\
&=
\frac{
 2\(1-\frac{1}{p^2}\)(1 + \frac{1}{p} + \cdots + \frac{1}{p^{k-1}})
  + \frac{1}{p^k} \cancel{\sum_{\ve_p \in \{\pm 1\}} \ve_p  \(1 - \frac{1}{p^2} \)}
}
{2\cdot \[ \(1-\frac{1}{p^2}\)(1 + \frac{1}{p} + \cdots + \frac{1}{p^{k-1}}) \]} 
\quad = \quad 1.
\end{align*}
}
%
%
\end{proof}

\section{Closing Remarks}

Since the proof of Theorem \ref{Thrm:main_theorem} almost exclusively involves local computations, it is reasonable to ask how it can be generalized.  We propose the following easily stated


\newcommand{\p}{\mathfrak{p}}

\begin{conj}
Suppose that $F$ is a totally real number field of class number one in which 2 is inert, $A_F$ is the ring of integers of $F$,  $2S(A_F^\times)^2$ is the squarefree determinant squareclass of a totally positive definite $A_F$-valued binary quadratic form, and $W \in A_F$ divides $S$.  Then
$$
\sum_{Q\in \text{$S$-genus}} \epsilon_W(Q) \frac{r_{Q}(m)}{|\Aut(Q)|} 
= \kappa_F\,W^{[F:\mathbb{Q}]}  \, \frac{r_{x^2 + y^2 + z^2}(m/W^2)}{|\Aut(x^2 + y^2 + z^2)|}
$$
for all totally positive $m \in A_F$ with $W \mid m$ and
 $m \in (A_F^\times)^2 \cup 2(A_F^\times)^2 \pmod {4A_F}$, where  $\epsilon_W(Q)$ is defined using quadratic norm residue symbols by the formula
 $$
\epsilon_W(Q) := \prod_{\p\mid W} \ve_\p(\Gen(Q)),
 $$ 
and $\kappa_F \in \Q>0$ is a constant depending only on $F$.
\end{conj}

\begin{proof}[Sketch of proof]
We proceed exactly as in the proof of Theorem \ref{Thrm:main_theorem}, with the following modifications:
\begin{enumerate}

\item The description of the binary genera of discriminant $-8S$  since as before we fixed the archimedean type of all forms in the $S$-genus, and there is still only one remaining place at $p=2$, whose local type is determined by the product formula.

\item Non-archimedean local density and local mass factors formulas at unramified places $\p\mid p$ of $F$ are the same if we replace factors of $p$ by $q:=|A_F/\p A_F|$.

\item Calculations with the mass formula will only differ in the Gamma function factors appearing, which contribute to the constant $\kappa_F$.

\item Lemma \ref{lemma:cases} remains unchanged if we replace odd primes $p$ by places $\p \nmid 2$.

\item Siegel's product formula and an explicit mass formula still holds over totally real number fields.
\end{enumerate}
To determine the exact constant $\kappa_F$ we need use an explicit 
mass formula valid for totally definite  ternary quadratic forms of level $4S  A_F$ over totally real number fields $F$. 
\end{proof}

\begin{rem}
While almost surely a version of the $S$-genus identities hold over an arbitrary totally real number field, the formulation becomes a little more involved.  In addition to there being a lot of additional work to control the contributions from the local factors at $\p\mid 2$, the presence of non-trivial class number makes the language of quadratic lattices (which may not be free as $A_F$-modules) the more natural framework in which to state the relevant identities.  
\end{rem}

\bibliographystyle{alpha}

\bibliography{refs}

\end{document}